\documentclass[10pt]{amsart}
\usepackage{latexsym,amsmath,amsthm, amssymb}

\newtheorem{thm}{Theorem}

\newtheorem{lemma}{Lemma}[section]
\newtheorem{theorem}[lemma]{Theorem}
\newtheorem{corollary}[lemma]{Corollary}

\newcommand{\GL}{\mathrm{GL}}

\newcommand{\PSU}{\mathrm{PSU}}
\newcommand{\GU}{\mathrm{GU}}
\newcommand{\PSp}{\mathrm{PSp}}

\newcommand{\gen}[1]{\ensuremath{\langle #1 \rangle}}
\newcommand{\syl}[2]{\mbox{{\rm Syl}$_{#1}(#2)$}}
\renewcommand{\o}[2]{\mbox{${\mathcal O}_{#1}(#2)$}}
\renewcommand{\c}[2]{\ensuremath{C_{#1}(#2)}}

\newcommand{\f}[1]{\ensuremath{F(#1)}}
\newcommand{\z}[1]{\ensuremath{Z(#1)}}
\newcommand{\ob}[1]{\ensuremath{\overline{#1}}}
\newcommand{\normal}{\unlhd}
\newcommand{\isom}{\cong}
\newcommand{\sfit}[1]{\ensuremath{\psi(#1)}}

\newcommand{\calC}{\mathcal{C}}
\newcommand{\fh}[1]{\ensuremath{f(#1)}}

\DeclareMathOperator{\codim}{codim}
\providecommand{\MR}[1]{}

\author{Paul Flavell}
\address{U  School of Mathematics, University of Birmingham,
Birmingham B15 2TT, UK}
\email{ P.J.Flavell@bham.ac.uk}

\author{Simon Guest}
\address{Department of Mathematics, USC, Los Angeles, CA 90089-2532, USA}
\email{sguest@usc.edu}

\author{Robert Guralnick}
\address{Department of Mathematics, USC, Los Angeles, CA 90089-2532, USA}
\email{guralnic@usc.edu}

\thanks{The last two authors were partially supported by
the NSF grant DMS
0653873.}


\date{\today}

\begin{document}
\title{Characterizations of the Solvable Radical}
\maketitle

\begin{abstract}
We prove that there exists a constant $k$ with the property: if
$\calC$\ is a conjugacy class of a finite group $G$ such that every
$k$ elements of $\calC$\ generate a solvable subgroup then $\calC$\
generates a solvable subgroup. In particular, using the
Classification of Finite Simple Groups, we show that we can take
$k=4$. We also present proofs that do not use the Classification
theorem. The most direct proof gives a value of $k=10$. By
lengthening one of our arguments slightly, we obtain a value of
$k=7$.
\end{abstract}
\section{Introduction}
The aim of this paper is to prove the following theorem:
\begin{thm} \label{thmA}
    There exists a constant $k$ with the property: if $\calC$\ is a conjugacy class of the
finite group $G$ such that every $k$ elements of $\calC$\ generate a solvable subgroup
then $\calC$\ generates a solvable subgroup.
\end{thm}

The most direct proof of Theorem~\ref{thmA} uses a value of $k =
10$. The ideas involved are similar to those used in the proofs of
Hall's extended Sylow Theorems together with a little
representation theory. After a preprint containing a proof of
Theorem~\ref{thmA} was circulated, Gordeev {\em et al.}
\cite{gordeev} used the Classification of Finite Simple Groups to
prove Theorem~\ref{thmA} with a value of $k = 8$. By lengthening
one of our arguments we are able to obtain a classification free
proof with a value of $k = 7$. Using deeper representation theory,
better results are possible, see \cite{arpf} and \cite{pf}. As
conjectured by Gordeev {\em et al.} \cite{gordeev}, which has
since been announced by them independently (see \cite{GGP1, GGP2}, 
we prove that, in
fact, we can take $k=4$. Both proofs for $k=4$  rely on the Classification
theorem.  See also \cite{GPS}.  Note also that $4$ is best possible
(consider the conjugacy class of transpositions in $S_n, n > 4$).

 For many conjugacy classes, it is worth noting that
Theorem~\ref{solvablebs} below implies that it is enough to
consider pairs of elements in $\calC$.
\begin{theorem}[\cite{Guest}] \label{solvablebs}
    Let $\calC$\ be a conjugacy class of the finite group $G$ consisting of elements of prime order
$p \ge 5$.Then $\calC$\ generates a solvable subgroup if and only if every pair of elements of $\calC$\ generates a solvable subgroup.
\end{theorem}

The corresponding result for nilpotency is true without any restriction on $p$.
This is the Baer-Suzuki theorem and is well known and reasonably
elementary.

These results do not hold for all groups, but the finite case does yield the
following results.

\begin{corollary}  \label{linear}  Let $k$ be a field and $G$ a subgroup
 of $\GL(n,k)$.
\begin{enumerate}
\item If $g \in G$, then the normal closure of $g$ in $G$ is solvable
if and only if every $4$ conjugates of $g$ generate a solvable subgroup.
\item If $k$ has characteristic $0$ or $p > 3$ and $g \in G$ is a unipotent
element, then the normal closure of $g$ in $G$ is solvable
if and only if every $2$ conjugates of $g$ generate a solvable subgroup.
\end{enumerate}
\end{corollary}

\section{Proof of Theorem A using the Classification theorem}

Let $(\calC,G)$ be a minimal counterexample. Then every four
elements of $\calC$\ generates a solvable subgroup, yet the subgroup
generated by $\calC$\ is not solvable. Since $|G|$ is minimal, it is
clear that the solvable radical of $G$ is trivial.  The following
lemma \cite[Lemma]{Guest} will be used to show that $G$ must be almost simple.

\begin{lemma} \label{reduction}
Suppose that $G$ is a finite group such that the Fitting subgroup
$\f{G}$ is trivial. Let $L$ be a component of $G$. \\
{\rm (a)} If $x$ is an element of $G$ such that $x \not\in N_G(L)$
and $x^2 \not\in C_G(L)$ then there exists an element $g$ in $G$
such
that $\left\langle x,x^g \right\rangle$ is not solvable. \\
{\rm (b)} If  $x$ is an element of $G$ such that $x \not\in
N_G(L)$ and $x^2 \in C_G(L)$ then there exist elements $g_1$ and
$g_2$ in $G$ such that $\left\langle x,x^{g_1}, x^{g_2}
\right\rangle$ is not solvable.
\end{lemma}

Part (a) of Lemma~\ref{reduction} relies on the so called
$\tfrac{3}{2}$-generation result of Guralnick and Kantor
\cite{GK}. Part $(b)$ of Lemma~\ref{reduction} only relies on the
fact that every finite simple group can be generated by two elements
(see \cite{GurAs}).   We can now show that $G$
is almost simple:

\begin{lemma}
 $G$ is almost simple.
\end{lemma}

\begin{proof}  Let $x \in \calC$ such that every four
conjugates of $x$ generate a solvable subgroup of $G$
but that $M:=\langle x^G \rangle$ is not solvable.

Let $N$ be a minimal normal subgroup of $G$.
Since $G$ has no solvable normal subgroups,
$N = L \times \ldots \times  L$ with $L$ a nonabelian
simple group.

By minimality, $MN/N$ is solvable.  If $[x,N]=1$,
then $[M,N]=1$ and so $M$ embeds in $G/N$, whence
$M$ is solvable.

Set $H=\langle x, N \rangle$.  The normal closure
of $x$ in $H$ contains $[x,N]$ which is a nontrivial
normal subgroup of $N$, whence is not solvable.
Thus, by minimality, $G=H$, whence $x$ acts transitively
on the direct factors of $N$.  By Lemma \ref{reduction},
this implies that $N=L$ is simple.  Since $C_{\langle x \rangle}(N)$
is central in $G$, this is trivial, whence
$N$ is the unique minimal normal subgroup of $G$.
Thus, $G$ is almost simple.
\end{proof}

Let $G_0$ be the socle of the almost simple group $G$. The
Classification of Finite Simple Groups imples that $G_0$ is an
alternating group, a simple group of Lie type, or a sporadic
group. Since the solvable radical of $G$ is trivial and
$(\calC,G)$ is a counterexample to the theorem, every four
elements of $\calC$ generate a solvable group. Observe that it
suffices to assume that the elements of $\calC$\ have prime order.
Indeed, the following theorem implies that we may assume that
$\calC$\ is a conjugacy class of involutions. It also precludes the
vast majority of possibilities for $G_0$. \par
\begin{theorem}[\cite{Guest}] \label{bsas}
Let G be a finite almost simple group with socle $G_0$. Suppose
that $x$ is an element of odd prime order in $G$. Then one
of the following holds. \\
  {\rm (i)} There exists $g \in G$ such that $\langle x , x^g \rangle $ is
  not solvable. \\
  {\rm (ii)} $x^3=1$ and $(x,G_0)$ belongs to a short list of exceptions given
  in Table \ref{exceptions}. Moreover, there exist $g_1, g_2 \in G$
  such that $\langle x , x^{g_1}, x^{g_2} \rangle$ is not solvable,
  unless $G_0 \cong PSU(n,2)$ or $PSp(2n,3)$. In any case, there exist
  $g_1,g_2,g_3 \in G$ such that
  $\langle x , x^{g_1}, x^{g_2}, x^{g_3} \rangle$ is not solvable.\\
\end{theorem}
\begin{table}
\begin{center}
  \begin{tabular}{|c|c|}
    \hline
    $G_0$ & $x$ \\
    \hline
    $PSL(n,3)$ & transvection  \\
    $PSp(2n,3)$ & transvection \\
    $PSU(n,3)$ & transvection \\
    $PSU(n,2)$ & reflection of order $3$\\
    $P \Omega^{\epsilon}(n,3)$ & $x$ a long root element \\
    $E_l(3),F_4(3), {^2}E_6(3), {^3}D_4(3)$ & $x$ a long root element \\
    $G_2(3)$ & $x$ a long or short root element \\
    $G_2(2)^{\prime} \cong PSU(3,3)$ & transvection \\
    \hline
  \end{tabular}
\vspace{5mm} \caption{List of exceptions to Theorem~\ref{bsas}}
\label{exceptions}
\end{center}
\end{table}

 Now observe that if $\langle x,x^g\rangle$ is a
$2$-group for all $g \in G$ then $\left\langle x^G\right\rangle$
is nilpotent by the Baer--Suzuki Theorem. So if $x$ is an
involution in an almost simple group there must exist a conjugate
$x^{g_1}$ such that $\langle x,x^{g_1}\rangle$ is not a $2$-group.
Thus, $x$ inverts an element $y$ of odd prime order.
So $(\calC, G)$ cannot be a minimal counterexample
unless:
\begin{enumerate}
 \item $G_0$ is one of the group in Table \ref{exceptions};
\item  $\calC$ is a conjugacy class of involutions; and
\item if $x \in \calC$, then
$x$ inverts no elements of odd prime order other than those in the listed conjugacy classes
of Table \ref{exceptions}.
\end{enumerate}

    We shall
rule out these possibilities case by case. \\

\noindent Case 1.  $G_0 = A_n(3), n > 1$.\\

Then $x$ inverts a transvection $y$.   This implies that $x$
normalizes the parabolic subgroup $P:=C(y)$.  If $n > 2$, this
implies that $x$ is inner-diagonal (for a graph automorphism
does not preserve the $G_0$ class of $P$).  If $x$ is inner diagonal,
then we may view it in $\GL(n+1,3)$, and since it has a nontrivial
eigenvalue over the field of three elements (it preserves a hyperplane),
we see that we may reduce to the case $n=2$, where it is clear that
$x$ inverts a regular unipotent element of order $3$.

So it remains to consider $x$ a graph automorphism with $n=2$.
There is a unique such class of involutions and we see that it inverts
an element of order $13$. \\

\noindent Case 2.  $G_0 = C_n(3), n > 1$.   \\

Again, $x$ inverts a transvection $y$, whence $x$ must be an
outer involution.  Then $x$ acts on $Q$, the unipotent radical
of $C(y)$.  Note that $Q$ is extraspecial.  Thus, $x$ cannot
centralize $Q/Z(Q)$ (since it does not centralizes $y \in Z(Q)$).
So $x$ must invert some element of $Q \setminus{Z(Q)}$.  However
all transvections in $Q$ are central, a contradiction.  \\

\noindent Case 3.  $G_0$ an orthogonal group over the field of
$3$ elements (of dimension at least $7$). \\

Any such involution $x$ can be viewed as acting on the natural orthogonal
module.  It is straightforward (since $x$ acts quadractically on
the module) to see that $x$ leaves invariant a nondegnerate subspace
of dimension $d=5$ or $6$ on which $x$ acts noncentrally.  The result
follows by induction. \\

\noindent Case 4.  $G_0 = \PSU(n,3), n > 2$.  \\

Any such involution $x$ can be viewed as acting (possibly semilinearly)
on the natural module with $x$ inverting a transvection $y$.
Let $Y$ be the group generated by $y$.  Then the normalizer of $Y$
is a parabolic subgroup $P$ with an extraspecial unipotent radical
$Q$.  Since $Y = Z(Q)$, it follows that $x$ cannot act as a scalar
on $Q/Z(Q)$.  Thus, $x$ is not in the solvable radical of $N_G(P)$
unless $n =3$ or $4$ (in which case $P$ is solvable).

If $n=3$ or $4$, it follows by \cite{GS} (or a straightforward
computation) that $4$ conjugates of $x$ generate a subgroup
containing $G_0$. \\

\noindent Case 5.  $G_0=F_4(3), E_{\ell}(3)$, or ${^2}E_6(3)$. \\

This is essentially the same argument as the previous case.
$x$ inverts $y$ and so acts on $P$, the normalizer of the a long
root subgroup $Y = \langle y \rangle$.  Note that the unipotent
radical $Q$ of $P$ is extraspecial, and so $x$ cannot act trivially
on $Q/Y$.   If $x$ is not in the solvable radical of $P$, the
result follows by induction.  If $x$ is in the solvable
radical of $P$, then $P/Q$ must centralize $x$ acting on
$Y/Q$, and so $x$ must act as inversion on $Q/Y$, whence
it centralizes $Y$, a contradiction.  \\

\noindent Case 6.  $G_0 = \PSU(n,2), n > 2$. \\

Any such involution $x$ can be viewed as acting (possibly semilinearly)
on the natural module with $x$ inverting a psuedoreflection $y$.
Thus, $x$ leaves invariant the fixed hyperplane of $y$, and so
$x$ embeds in the normalizer of $\GU(n-1,2)$.  It clearly is not central
on the hyperplane and so not in the solvable radical unless $n=4$.
Since $\PSU(4,2)=\PSp(4,3)$, this is a case we have already dealt with. \\

\noindent Case 7.  $G_0 = {^3}D_4(3)$.  \\

It follows by \cite{MSW} that there are three conjugates of $x$
generating $G_0$. \\

\noindent Case 8.  $G_0 = G_2(3)$. \\

If $x$ is inner, it follows by \cite{MSW} that three conjugates
of $x$ generate $G_0$.  If $x$ is outer, then $x$ interchanges
short root elements and long root elements and so cannot invert
either type of element, whence the result holds in this case. \\

We have now dealt with all cases, and so the proof is complete.

\section{A classification free approach}
The purpose of this section is to explore what can be proved using
only elementary means. \par

For a prime $q$ we let \o{q}{G} be the largest normal $q$-subgroup
of $G$. If $G\not=1$ is solvable we let \fh{G} denote the Fitting
height of $G$. This is the smallest integer $n$ such that $G$
possesses a series
\[
    1=F_{0}\normal F_{1}\normal \cdots \normal F_{n}=G
\]
with $F_{i+1}/F_{i}$  nilpotent for all $i$. The trivial group has
Fitting height 0; a nontrivial nilpotent group has Fitting height
1; and if $G\not=1$ then $\fh{G/\f{G}}=\fh{G}-1$.

If $G\not=1$ is solvable we define
\[
    \sfit{G}\ =\ \bigcap\{ K\normal G\ |\ \fh{G/K}<\fh{G} \}.
\]
Now $G/\sfit{G}$ is isomorphic to a subgroup of a direct product
of groups each with Fitting height less than \fh{G}. Thus
$\fh{G/\sfit{G}}<\fh{G}$. It follows that
\[
    1\not=\sfit{G}\leq\f{G}
\]
and that \sfit{G} is the unique smallest normal subgroup of $G$
such that the corresponding quotient group has Fitting height less
than the Fitting height of $G$.

\begin{lemma}\label{p1}
    Let $H$ be a subgroup of the solvable group $G\not=1$.
    If $\fh{H}=\fh{G}$ then
    \[
        \sfit{H}\leq\sfit{G}\leq\f{G}.
    \]
\end{lemma}
\begin{proof}
    Set $\ob{G}=G/\sfit{G}$, so that $\fh{\ob{G}}<\fh{G}$.
    Then $\fh{\ob{H}}\leq\fh{\ob{G}}<\fh{G}=\fh{H}$ so
    the definition of \sfit{H} implies that $\sfit{H}\leq\sfit{G}$.
    We have already seen that $\sfit{G}\leq\f{G}$.
\end{proof}

\begin{lemma}\label{p2}
    Let $G$ be a solvable group, let $N\normal G$,
    set $\ob{G}=G/N$ and suppose that $\ob{G}\not=1$.
    Then the following are equivalent:
    \begin{enumerate}
        \item[(i)]   $\ob{\sfit{G}}\not=1$.
        \item[(ii)]   $\fh{\ob{G}}=\fh{G}$.
        \item[(iii)]   $\sfit{\ob{G}}=\ob{\sfit{G}}$.
    \end{enumerate}
\end{lemma}

\begin{proof}
    Suppose that $\ob{\sfit{G}}\not=1$.
    Then $\sfit{G}\not\leq N$ so the definition of \sfit{G}
    implies that $\fh{\ob{G}}=\fh{G}$.
    Thus (i) implies (ii).

    Suppose that $\fh{\ob{G}}=\fh{G}$.
    Now $\ob{G}/\ob{\sfit{G}}$ is a homomorphic image of $G/\sfit{G}$ so
    $\fh{\ob{G}/\ob{\sfit{G}}}\leq\fh{G/\sfit{G}}<\fh{G}=\fh{\ob{G}}$
    whence $\sfit{\ob{G}}\leq\ob{\sfit{G}}$.
    Let $K$ be the full inverse image of \sfit{\ob{G}} in $G$.
    Then $G/K\isom\ob{G}/\sfit{\ob{G}}$ so $\fh{G/K}<\fh{\ob{G}}=\fh{G}$
    whence $\sfit{G}\leq K$ and then $\ob{\sfit{G}}\leq\ob{K}=\sfit{\ob{G}}$.
    We deduce that $\sfit{\ob{G}}=\ob{\sfit{G}}$.
    Thus (ii) implies (iii).

    Since $\ob{G}\not=1$ we have $\sfit{\ob{G}}\not=1$ so (iii) implies (i).
\end{proof}

\bigskip

\begin{lemma}\label{p3}
    Suppose that the solvable group $G$ possesses a unique minimal normal subgroup $V$. Then $V$ acts transitively by conjugation on the set of complements to $V$ in $G$.
\end{lemma}

\begin{proof}
    Suppose that $A$ is a complement to $V$ and let $Q$ be a minimal normal subgroup of $A$, so that $Q$ is a $q$-group for some prime $q$. Set $K=QV$. Now $\c{V}{Q}=1$ since otherwise $Q$ would be another minimal normal subgroup of $G$. It follows that $V$ is an $r$-group for some prime $r\not=q$. Then $Q$ is a Sylow $q$-subgroup of $K$ and any complement to $V$ in $G$ is the normalizer of a Sylow $q$-subgroup of $K$. The result now follows from Sylow's Theorem.
\end{proof}

The following  extends a result that appears in \cite[page
82]{mw}.

\begin{lemma}\label{t1}
    Let $G$ be a solvable group that possesses an element $a$
    such that $G=\gen{a^{G}}$.  Let $k$ be a field.
    Let $V$ be a nontrivial irreducible $kG$-module.
    Then
    \[
        \dim\c{V}{a}\ \leq\ \frac{3}{4}\dim V.
    \]
\end{lemma}

\begin{proof}  First note that by replacing $k$
by $\mathrm{End}_{kG}(V)$, we may assume that $V$ is absolutely
irreducible.  Then we can extend scalars and assume that
$k$ is algebraically closed.  Clearly, we may assume that
$G$ acts faithfully on $V$.

    Assume false, so that $\dim\c{V}{a}>\frac{3}{4}\dim V$.
    We will construct a normal subgroup of $G$ that
    has more than one homogeneous component on $V$.
    The proof then proceeds by analyzing the permutation action
    of $G$ on those components.

    Let $g,h\in G$.
    The subspaces \c{V}{a} and \c{V}{a^{g}} both have
    dimension greater than $\frac{3}{4}\dim V$
    so their intersection has dimension greater than $\frac{1}{2}\dim V$.
    Since $[g,a]$ acts trivially on this intersection it follows that
    \[
        \dim\c{V}{[g,a]}\ >\ \frac{1}{2}\dim V.
    \]
    Repeating this argument,
    we deduce that
    \begin{equation} \label{eq1}
        \c{V}{[g,a]}\not=0\ \ \mbox{and}\ \ \c{V}{[g,a,h]}\not= 0
    \end{equation}
    for all $g,h\in G$.

    Since $G$ acts irreducibly and faithfully on $V$ we have
    \begin{equation}\label{eq2}
        \c{V}{z} = 0
    \end{equation}
    for all $z\in\z{G}^{\#}$.
    In particular, $a\not\in Z(G)$.
    Let $N$ be a normal subgroup of $G$ chosen minimal subject to $N$ is not central
in $G$.  Now $N$ is solvable, so $N' < N$, whence $N' \le Z(G)$.   We claim
that $N$ is abelian.  If not, then $Z(N) < N$, whence $Z(N) \le Z(G)$.

    Since $G=\gen{a^{G}}$ and $N$ is not central, we see that $[a,N] \ne 1$.
    Choose $g\in N$ such that $[g,a]\not=1$.  Since $[g,a]$ fixes a nonzero
vector in $V$, $[g,a]$ is a noncentral element of $N$.  Now choose $h \in N$
with $1 \ne [g,a,h] \in N' \le \z{G}$.  Thus, $[g,a,h]$ is a nontrivial scalar
on $V$, but by Equations \ref{eq1} and \ref{eq2}, this is not the case.  Thus, $N$ is abelian.

    Since $N$ is abelian and  not central in $G$,
$V=V_1 \oplus \ldots \oplus V_r$ is a direct sum of the $N$ eigenspaces
$V_i$ with $r > 1$.  Set $\Omega=\{V_1, \ldots, V_r\}$.
Since $V$ is irreducible, $G$ acts transitively on $\Omega$.
Since $G$ is generated by the conjugates of $a$, $a$ acts nontrivially
on $\Omega$.  Set $e = \dim V_i$.

We claim that $a$ fixes no more than $d/2$ points in any
transitive permutation action of $G$ of degree $d > 1$.  It suffices
to prove this for a primitive action  (if $a$ fixes no more than
$1/2$ the blocks, it fixes no more than $1/2$ the points).   In
any primitive action of $G$, $a$ acts nontrivially.  Since
a primitive permutation action of a finite solvable group consists
of affine transformations of a vector space over a prime field,
the claim follows.

If $\Delta$ is an $a$-orbit on $\Omega$ and $V_{\Delta} = \sum_{i \in \Delta} V_i$,
then $\dim C_{V_{\Delta}}(a) \le e$.  Thus,
$\dim C_V(a) \le e f$ where $f$ is the number of orbits of $a$ on $\Omega$.
Since $a$ fixes at most $r/2$ points, it has at most $3r/4$ orbits on $\Omega$,
whence $\dim C_V(a) \le (3/4) \dim V$, a contradiction.
\end{proof}

\begin{lemma}\label{t2}
    Let $G$ be a solvable group and let $a$ be an element of $G$ with prime order.
    Suppose that $A$ is a subgroup of $G$ with the following properties:
    \begin{enumerate}
        \item[(i)]   $A=\gen{a_{1},\ldots,a_{5}}$ where $a_{1},\ldots,a_{5}$ are
            conjugate to $a$ in $G$ and conjugate to one another in $A$.
        \item[(ii)]   $A$ has maximal Fitting height subject to (i).
    \end{enumerate}
    Then
    \[
        \sfit{A}\leq\f{G}.
    \]
\end{lemma}

\begin{proof}
    Assume false and let $G$ be a minimal counterexample,
    so that $\sfit{A}\not\leq\f{G}$.
    Let $V$ be a minimal normal subgroup of $G$ and set
    \[
        \ob{G} = G/V.
    \]
    Now $V$ is abelian so $V\leq\f{G}$.
    In particular, $\sfit{A}\not\leq V$ and then
    the definition of \sfit{A} implies that
    \begin{equation}\label{11}
        \fh{\ob{A}}=\fh{A}.
    \end{equation}

    We claim that \ob{A} satisfies (i) and (ii) when $G$ is replaced by \ob{G} and $a$ by \ob{a}.
    Certainly (i) is satisfied.
    As for (ii), let \ob{B} be a subgroup of \ob{G} such that
    $\ob{B}=\gen{\ob{b}_{1},\ldots,\ob{b}_{5}}$ with
    $\ob{b}_{1},\ldots,\ob{b}_{5}$ conjugate to
    \ob{a} in \ob{G} and conjugate to one another in \ob{B}.
    Let $b_{1}$ be a conjugate of $a$ that maps onto $\ob{b}_{1}$ and
    let $B$ be an inverse image of \ob{B} that is minimal subject to $b_{1}\in B$.
    Choose $g_{2},\ldots,g_{5}\in B$ such that $\ob{b}_{1}^{\ob{g}_{i}}=\ob{b}_{i}$.
    Then \gen{b_{1},b_{1}^{g_{2}},\ldots,b_{1}^{g_{5}}} is a subgroup of $B$
    that maps onto \ob{B}.
    The minimality of $B$ forces $B=\gen{b_{1},b_{1}^{g_{2}},\ldots,b_{1}^{g_{5}}}$
    and as $g_{2},\ldots,g_{5}\in B$ we see that
    $B$ is a subgroup of $G$ that satisfies (i).
    Consequently
    \[
        \fh{A}\geq\fh{B}.
    \]
    Now \ob{B} is a homomorphic image of $B$ so $\fh{B}\geq\fh{\ob{B}}$ and
    then using (\ref{11}) we have
    \[
        \fh{\ob{A}}\geq\fh{\ob{B}}.
    \]
    This proves the claim.

    The minimality of $G$ and the previous paragraph
    imply that $\sfit{\ob{A}}\leq\f{\ob{G}}$.
    Lemma~\ref{p2} and (\ref{11}) imply that $\sfit{\ob{A}}=\ob{\sfit{A}}$
    so we deduce that
    \begin{equation}\label{12}
        \ob{\sfit{A}}\ \leq\ \f{\ob{G}}.
    \end{equation}
    It follows readily that $V$ is the unique minimal normal subgroup of $G$.
    Indeed, if $U$ were another such subgroup then \gen{\sfit{A}^{G}} would embed into
    the nilpotent group $\f{G/U}\times\f{G/V}$,
    contrary to the fact that $\sfit{A}\not\leq\f{G}$.

    Since \sfit{A} and \f{G} are nilpotent and since $\sfit{A}\not\leq\f{G}$,
    there exists a prime $q$ such that $\o{q}{\sfit{A}}\not\leq\o{q}{G}$.
    Set $Q=\o{q}{\sfit{A}}$.
    By (\ref{12}) we have $\ob{Q}\leq\o{q}{\ob{G}}$.
    Let $K$ be the full inverse image of \o{q}{\ob{G}} in $G$,
    so that $Q\leq K\normal G$ and $K/V$ is a $q$-group.
    Now $G$ is solvable so $V$ is an elementary abelian $r$-group for some prime $r$.
    Moreover, $Q\not\leq\o{q}{G}$ so $K$ is not a $q$-group and hence $r\not=q$.
    Since $V$ is the unique minimal normal subgroup of $G$ we deduce that $\o{q}{G}=1$.

    We claim that
    \begin{equation}\label{13}
        \c{K}{V}\ =\ V.
    \end{equation}
    Indeed, choose $S\in\syl{q}{K}$.
    Since $K/V$ is a $q$-group we have $K=SV$ whence $\c{K}{V}=\c{S}{V}\times V$.
    Then $\c{S}{V}=\o{q}{\c{K}{V}}\leq\o{q}{K}\leq\o{q}{G}=1$,
    proving the claim.

    Suppose that $AV\not=G$.
    Then the minimality of $G$ implies that $Q\leq\o{q}{AV}$ so
    as $V\leq\o{r}{AV}$ we see that $[Q,V]=1$,
    contrary to (\ref{13}).
    We deduce that
    \begin{equation}\label{14}
        G\ =\ AV.
    \end{equation}

    If $\fh{A}=\fh{G}$ then Lemma~\ref{p1} implies that $Q\leq\o{q}{G}$,
    contrary to the choice of $q$.
    Thus $\fh{A}<\fh{G}$ and then the definition of $A$ implies that $G$ cannot be
    generated by 5 conjugates of $a$.
    Moreover, we have $A\not=G$ so using (\ref{14}) and
    the fact that $V$ is a minimal normal subgroup of $G$ we deduce that $A\cap V=1$,
    so $A$ is a complement to $V$ in $G$.

    Let $u_{1},\ldots,u_{5}\in V$ and
    set $C=\gen{a_{1}^{u_{1}},\ldots,a_{5}^{u_{5}}}$.
    Since $A=\gen{a_{1},\ldots,a_{5}}$ and
    since $G=AV$ we have $G=CV$.
    Now $G$ cannot be generated by 5 conjugates of $a$ and
    $V$ is a minimal normal subgroup of $G$ so $C\cap V=1$.
    In particular, $C$ is a complement to $V$.
    By Lemma~\ref{p3} there exists $v\in V$ such that $C^{v}=A$.
    Thus
    \[
        \gen{a_{1}^{u_{1}v},\ldots,a_{5}^{u_{5}v}}\ =\ A\ =\  \gen{a_{1},\ldots,a_{5}}.
    \]
    For each $i$ we have $u_{i}v\in V\normal G$ so
    \[
        [a_{i},u_{i}v] = a_{i}^{-1}a_{i}^{u_{i}v} \in A\cap V = 1,
    \]
    whence $u_{i}v\in\c{V}{a_{i}}$ and then $u_{i}\in\c{V}{a_{i}}v$.
    This proves that the natural map
    \[
        V\longrightarrow V/\c{V}{a_{1}}\times\ldots\times V/\c{V}{a_{5}}
    \]
    is surjective.
    Since all of the $a_{i}$ are conjugate to $a$ we deduce that
    \begin{equation}\label{15}
        \dim V\ \geq\ 5\codim\c{V}{a}.
    \end{equation}

    Now $V$ is an elementary abelian normal subgroup of $G$ so
    $V$ may be regarded as an $A$-module.
    Since $G=AV$ and since $V$ is a minimal normal subgroup of $G$
    we see that $V$ is an irreducible $A$-module.
    It follows from (\ref{13}) that the action of $A$ on $V$ is nontrivial.
    Since $A=\gen{a_{1}^{A}}$ and since $a_{1}$ is conjugate to $a$,
    we may apply Lemma~\ref{t1} to conclude that
    \[
        \codim\c{V}{a}\ \geq\ \frac{1}{4}\dim V.
    \]
    This contradicts (\ref{15}) and completes the proof of this lemma.
\end{proof}

\bigskip

\noindent The following lemma proves Theorem~A.

\begin{lemma}\label{t3}
    Let $\calC$\ be a conjugacy class of the group $G$.
    If every 10 members of $\calC$\ generate a solvable subgroup then
    $\calC$\ generates a solvable subgroup.
\end{lemma}

\begin{proof}
    Assume false and let $G$ be a minimal counterexample.
    Then $G$ possesses no nontrivial normal solvable subgroups and
    we may suppose that the elements of $\calC$\ have prime order.
    Let $a\in\calC$ and let $A$ be a subgroup of $G$ that satisfies
    \begin{enumerate}
        \item   $A=\gen{a_{1},\ldots,a_{5}}$ where $a_{1},\ldots,a_{5}$ are conjugate
                to $a$ in $G$ and conjugate to one another in $A$, and
        \item   $A$ has maximal Fitting height subject to (i).
    \end{enumerate}
    Replacing $A$ by a suitable conjugate, we may suppose that $a_{1}=a$.
    Let $Q=\sfit{A}$, so that $Q\not=1$.
    Let $g\in G$ and set $H=\gen{A,A^{g}}$.
    By hypothesis, $H$ is solvable so Lemma~\ref{t2} with $H$ in place of $G$,
    yields $Q\leq\f{H}$.
    Similarly, $Q^{g}\leq\f{H}$.
    We deduce that \gen{Q,Q^{g}} is nilpotent for all $g\in G$.
    The Baer--Suzuki Theorem implies that \[
        Q\ \leq\ \f{G}. \]
    This contradicts the fact that $G$ has no nontrivial normal solvable subgroups
    and completes the proof.
\end{proof}

\noindent Using a slightly longer argument we are able to replace 10 by 7.
First we need:

\begin{lemma}\label{t3a}
    Suppose $A \not= 1$ is solvable and that $\sfit{A} \leq \z{A}$.
    Then $A$ is abelian.
\end{lemma}
\begin{proof}
    We have \[
        \fh{A/\sfit{A}} < \fh{A}.
    \]
    On the other hand, for any $n \geq 1$,
    the class if solvable groups of Fitting height $n$ is closed
    under central extensions.
    This forces $\fh{A/\sfit{A}} = 0$, whence $A = \sfit{A}$.
    As $\sfit{A} \leq \z{A}$,
    the conclusion follows.
\end{proof}
\begin{theorem}\label{t4}
    Let $\calC$\ be a conjugacy class of the group $G$.
    If every 7 members of $\calC$\ generate a solvable subgroup then
    $\calC$\ generates a solvable subgroup.
\end{theorem}

\begin{proof}
    Proceed as in the proof of the previous lemma and construct the subgroup $A$.

    We claim there is a prime $p$,
    a conjugate $b$ of $a$ and a $p$-subgroup $P$
    with $1 \not= P \leq \sfit{A} \cap \gen{a,b}$.
    If $[\sfit{A},a] \not= 1$ there exists a prime $p$
    and $x \in \o{p}{\sfit{A}}$ with $[a,x] \not= 1$.
    Put $b = a^{x}$ and $P = \gen{[a,x]}$.
    Suppose that $[\sfit{A},a] = 1$.
    As $A = \gen{a^{A}}$ it follows that $\sfit{A} \leq \z{A}$.
    The previous lemma implies that $A$ is abelian.
    Then $A = \gen{a}$.
    Put $b = a$ and $P = A$.

    Let $g \in G$ and set $H = \gen{A, a^{g}, b^{g}}$.
    By hypothesis, $H$ is solvable.
    Lemma~\ref{t2} implies $P \leq \o{p}{H}$.
    As $P^{g} \leq \gen{a,b}^{g} \leq H$ it follows that
    $\gen{P,P^{g}}$ is a $p$-group.
    A contradiction follows from the Baer--Suzuki Theorem.
\end{proof}

\section{Proof of the Corollary}

The proof of Corollary \ref{linear} is standard.  We first prove (1).
We first note the well known fact that if $H$ is a solvable subgroup
of $\GL(n,k)$, then the derived length of $H$ is bounded by a function
$f=f(n)$.

So suppose that the normal closure $N$ of $g$ in $H$ is not solvable.
Then there is some nontrivial element $x$ in the $f$th term in the derived
series of $N$.  We may pass to a subgroup of $G$ and assume that $G$
is finitely generated, and so $G \le GL(n,R)$ where $R$ is a finitely
generated ring over the prime field of $k$.  We can choose a maximal
ideal $M$ of $R$ such that $x$ is not in the congruence kernel of
the map $\phi: \GL(n,R) \rightarrow \GL(n,R/M)$.  Thus,
$\phi(N)$ is not solvable and $\phi(G)$ is finite, whence some four conguates of
$\phi(g)$ generate a nonsolvable subgroup.  Thus, the same is true for $G$.

The proof of (2) is essentially the same. First, as above, reduce to the case
that $G$ is finitely generated and contained in $\GL(n,R)$ where $R$ is a finitely generated
ring over $\mathbb{Z}$.   Now argue exactly as above (except that if the
characteristic is $0$, take $M$ to be a maximal ideal containing some
prime $p > 3$) and so our unipotent element in the image has order divisible
by the characteristic, a prime at least $5$.

\end{document}